\documentclass{amsart}
\usepackage{mathtools,amsthm}
\usepackage[letterpaper, margin=1.3in]{geometry}
\usepackage[utf8]{inputenc}
\usepackage[T2A,T1]{fontenc}
\usepackage{imakeidx}
\usepackage{enumitem}
\usepackage{mathrsfs}
\usepackage{tikz-cd}
\usepackage{amssymb}
\usepackage{xspace}
\usepackage{enumitem}
\usepackage{etoolbox}
\usepackage{stmaryrd}
\usepackage{setspace}
\usepackage{multirow}
\usepackage{pdflscape}
\usepackage{adjustbox}
\usepackage{nicematrix}
\usepackage[pagebackref]{hyperref}
\usepackage[nameinlink]{cleveref}
\usepackage{todonotes}

\hypersetup{colorlinks={true},linkcolor={blue},urlcolor=blue}

\makeindex[columns=2, title={Index of Notation},options={-s structural/index_style.ist}]
\makeatletter
\let\ori@idxitem\@idxitem
\def\@idxitem{\clear@penalties\ori@idxitem}
\def\clear@penalties{\subitem@count=3 }
\newcount\subitem@count
\def\subitem{%
  \advance\subitem@count -1
  \par
  \ifnum\subitem@count>0 \penalty10000 \fi
  \ori@idxitem%
}
\makeatother

\newtheorem{theorem}{Theorem}[subsection]
\newtheorem{lemma}[theorem]{Lemma}
\newtheorem{corollary}[theorem]{Corollary}
\newtheorem{proposition}[theorem]{Proposition}

\newtheorem{innercustomgeneric}{\customgenericname}
\providecommand{\customgenericname}{}
\newcommand{\newcustomtheorem}[2]{%
  \newenvironment{#1}[1]
  {%
   \renewcommand\customgenericname{#2}%
   \renewcommand\theinnercustomgeneric{\kern-0.3em ##1}%
   \innercustomgeneric
  }
  {\endinnercustomgeneric}
}

\newcustomtheorem{specialtheorem}{}
\crefname{specialtheorem}{}{}

\theoremstyle{definition}
\newtheorem{definition}[theorem]{Definition}

\theoremstyle{remark}
\newtheorem{remark}[theorem]{Remark}

\newcommand{\Z}{\mathbb{Z}}
\newcommand{\Q}{\mathbb{Q}}

\renewcommand{\C}{\mathbb{C}}

\newcommand{\defeq}{\vcentcolon=}

\newcommand{\isequal}{\stackrel{?}{=}}
\newcommand{\qbinom}[3]{\genfrac{[}{]}{0pt}{}{#1}{#2}_{#3}}

\begin{document}

\title[Unequal-parameter Kostka--Foulkes polynomials of type $C_n$]{Unequal-parameter Kostka--Foulkes polynomials of type $C_n$ at fundamental weights}
\author{Murilo Corato-Zanarella}

\begin{abstract}
We establish an explicit formula for the Kostka--Foulkes polynomial of type $C_n$ with unequal parameters at fundamental weights. As an application, we deduce an explicit Lusztig--Kato formula for the Hecke algebra of the special (but not hyperspecial) maximal compact subgroup of a quasi-split but non-split special orthogonal group over a $p$-adic field.
\end{abstract}
\maketitle
\setcounter{tocdepth}{2}
\let\oldtocsection=\tocsection
\let\oldtocsubsection=\tocsubsection
\let\oldtocsubsubsection=\tocsubsubsection
\renewcommand{\tocsection}[2]{\hspace{0em}\oldtocsection{#1}{#2}}
\renewcommand{\tocsubsection}[2]{\hspace{1em}\oldtocsubsection{#1}{#2}}
\renewcommand{\tocsubsubsection}[2]{\hspace{2em}\oldtocsubsubsection{#1}{#2}}
\tableofcontents

\section{Introduction}
For a simple Lie algebra $\mathfrak{g}$ and dominant weights $\lambda$ and $\mu,$ the Kostka--Foulkes polynomials $K^{\mathfrak{g}}_{\lambda,\mu}(t)$ (also called Lusztig's $t$-weight multiplicities) are ubiquitous objects in the study of the representation theory of $\mathfrak{g}.$ For example, the Lusztig--Kato formula \cite{Lusztig,Kato} interprets them as the transition coefficients between two natural bases of symmetric functions: the characters of highest weight representations and the Satake transforms of the Cartan basis.

When $\mathfrak{g}$ has roots of different lengths, there is a natural two-variable extension $K^{\mathfrak{g}}_{\lambda,\mu}(t,u)$ of these polynomials such that $K^{\mathfrak{g}}_{\lambda,\mu}(t)=K^{\mathfrak{g}}_{\lambda,\mu}(t,t).$ This article studies such two-variable Kostka--Foulkes polynomial for the Lie algebra of type $C_n.$ We note that this is different from the $q,t$-Kostka polynomials of \cite{Macdonaldqt}, and different from the double deformation in terms of parabolic subsystems of \cite{Lecouvey}.

\subsection{Main results}
Our main result is an explicit formula for the two-variable Kostka--Foulkes polynomial of type $C_n$ at fundamental weights.
\begin{specialtheorem}{Theorem A}[{\Cref{ThmA}}]
    We denote by $\omega_1,\ldots,\omega_n$ the fundamental weights of type $C_n.$ Denote $\omega_0=0.$ Then for $0\le a,b\le n,$ we have that $K^{C_n}_{\omega_a,\omega_b}$ is zero unless $a\ge b$ and $a\equiv b\mod 2.$ In that case, we have
    \begin{equation}
        K^{C_n}_{\omega_{a},\omega_b}(t,u)=t^{\frac{a-b}{2}}\frac{t^{n+1-a}-1}{t^{n+1-b}-1}\sum_{i=0}^{\frac{a-b}{2}}(-u)^i\left(t^{i^2-2i}\qbinom{n-b+1}{2i,\frac{a-b}{2}-i,n+1-\frac{a+b}{2}-i}{t}(t;t^2)_i\right).
    \end{equation}
\end{specialtheorem}

As an application, we deduce an explicit Lusztig--Kato formula for non-split orthogonal groups at fundamental weights. We refer to \Cref{SectionLK} for the definitions.
\begin{specialtheorem}{Theorem B}[{\Cref{ThmB}}]
    Let $F/\Q_p$ be a finite extension with residue field of cardinality $q,$ and let $V$ be a quadratic space over $F$ of discriminant $1\in\Q_p^\times/\Q_p^{\times2},$ Hasse invariant $-1$ and dimension $2n+2$ for some $n\ge1.$ For $0\le\delta\le n,$ we have
    \begin{equation}
        \sum_{\substack{\delta\le i\le n\\i\equiv\delta\mod2}}\frac{q^{\delta+1}-1}{q^{i+1}-1}\qbinom{i+1}{\frac{i-\delta}{2}}{q^2}\mathrm{Sat}(T_{n-i})=q^{\frac{1}{2}(n^2+n-\delta^2-\delta)}\sum_{\substack{\delta\le i\le n\\i\equiv\delta\mod2}}\frac{\delta+1}{i+1}\binom{i+1}{\frac{i-\delta}{2}}s_{n-i}(\boldsymbol{\mu}).
    \end{equation}
\end{specialtheorem}

\subsection{Notation on $t$-analogues}
For $n\in\Z_,$ we consider the $t$-analogues $[n]_t=\frac{1-t^n}{1-t}\in\Z[t,t^{-1}].$ This gives us the $t$-factorials $[n]_t!=\prod_{i=1}^n[i]_t\in\Z[t]$ for $n\in\Z_{\ge0},$ as well as the $t$-binomials $\qbinom{n}{a}{t}=\frac{\prod_{i=0}^{a-1}[n-i]_t}{[a]_t!}\in\Z[t,t^{-1}]$ for $n\in\Z,$ $a\in\Z_{\ge0}.$ Similarly, we consider the $t$-Pochhammer symbols $(a;t)_n=\prod_{i=0}^{n-1}(1-at^i)\in\Z[a,t]$ for $n\in\Z_{\ge0},$ as well as $(a;t)_\infty=\prod_{i\ge0}(1-at^i)\in\Z\llbracket a,t\rrbracket.$

\subsection*{Acknowledgments}
We want to thank C\'edric Lecouvey for helpful discussions related to Kostka--Foulkes polynomials.

\section{Computation of Kostka--Foulkes polynomial}
\subsection{Main formula}\label{KFsetup}
Let $W_n$ denote the Weyl group of type $C_n,$ and $R_n^+$ the set of positive roots. Denote $\Lambda_n=\Z^n$ the weight space, $\Lambda^+_n$ the subset of dominant weights, and $\Delta_n^+\subseteq\Lambda_n$ the subset of positive simple roots. We partition $R_n=R_n^s\sqcup R_n^l$ into short and long roots. We use similar notation such as $R_n^{+,s},R_n^{+,l},\Delta^{+,s}_n,\Delta^{+,l}_n.$ We write $\rho_n=\frac{1}{2}\sum_{\alpha\in R^+_n}\alpha$ for half the sum of positive roots, which belongs to $\Lambda^+_n$ in this case. We denote by $\ge$ the usual order on $\Lambda_n,$ namely $\mu\ge\lambda$ if $\mu-\lambda$ is a non-negative linear combination of positive roots. For $\lambda=(\lambda_1,\ldots,\lambda_n)\in\Lambda_n,$ we denote $\lvert\lambda\rvert\defeq\sum_{i=1}^n\lambda_i.$

Concretely, if $\Lambda_n=\Z e_1\oplus\cdots\Z e_n,$ then the roots are
\begin{equation}
    R_n^s=\{\pm e_i\pm e_j\text{ for }1\le i<j\le n\},\quad R_n^l=\{\pm 2e_i\text{ for }1\le i\le n\}
\end{equation}
with positive simple roots
\begin{equation}
    \Delta^+=\Delta^{+,s}\sqcup\Delta^{+,l}=\{e_1-e_2,\ e_2-e_3,\ldots,e_{n-1}-e_n\}\sqcup\{2e_n\}.
\end{equation}

We consider two variables $t,u,$ and for a root $\alpha\in R_n$ we denote
\begin{equation}
    t_\alpha=\begin{cases}
        t&\text{if $\alpha\in R_n^{s}$ is short,}\\
        u&\text{if $\alpha\in R_n^{l}$ is long.}
    \end{cases}
\end{equation}
\begin{definition}
    We consider the partition function $\mathcal{P}_n(\mu,t,u)\in\Z[t,u]$ for $\mu\in\Lambda_n$ characterized by
    \begin{equation}
        \prod_{\alpha\in R_n^+}(1-t_\alpha e^\alpha)^{-1}=\sum_{\mu\in\Lambda_n}\mathcal{P}_n(\mu,t,u)e^\mu,
    \end{equation}
    and the two-variable Kostka--Foulkes polynomial
    \begin{equation}
        K^{C_n}_{\lambda,\mu}(t,u)\defeq\sum_{w\in W_n}(-1)^w\mathcal{P}_n(w(\lambda+\rho_n)-(\mu+\rho_n),t,u).
    \end{equation}
    for $\lambda,\mu\in\Lambda^+_n.$ Note this is only nonzero when $\lambda\ge\mu.$
\end{definition}

Denote by $\omega_{k}=(\underbrace{1,\ldots,1}_k,\underbrace{0,\ldots,0}_{n-k})$ the fundamental weights for $0\le k\le n.$ The goal of this section is to prove the following theorem.
\begin{theorem}\label{KFthm}
    We have
    \begin{equation}
        K^{C_n}_{\omega_{2k},0}(t,u)=\frac{t^{n+1-k}-t^k}{t^{n+1}-1}\sum_{a=0}^k(-u)^a\left(t^{a^2-2a}\qbinom{n+1}{2a,k-a,n+1-k-a}{t}(t;t^2)_a\right).
    \end{equation}
\end{theorem}

\begin{corollary}\label{KFcor}
    We have
    \begin{equation}
        K^{C_n}_{\omega_{2k},0}(t,t)=\frac{t^{2(n+1-k)}-t^{2k}}{t^{2(n+1)}-1}\qbinom{n+1}{k}{t^2}
    \end{equation}
    and
    \begin{equation}
        K^{C_n}_{\omega_{2k},0}(t,t^2)=\frac{t^{n+1-k}-t^{k}}{t^{n+1}-1}\qbinom{n+1}{k}{t^2}.
    \end{equation}
\end{corollary}
\begin{remark}
    We note that the above formula for $K^{C_n}_{\omega_{2k},0}(t,t)$ was already known by \cite[Theorem 7.2]{Lecouvey-Lenart}.
\end{remark}
\begin{proof}
We need to prove that
\begin{equation}
    \sum_{a=0}^k\left((-1)^at^{a^2}\qbinom{n+1}{2a,k-a,n+1-k-a}{t}(t;t^2)_a\right)\isequal\qbinom{n+1}{k}{t^2}
\end{equation}
and 
\begin{equation}
    \sum_{a=0}^k\left((-1)^at^{a^2-a}\qbinom{n+1}{2a,k-a,n+1-k-a}{t}(t;t^2)_a\right)\isequal\frac{t^{n+1-k}+t^k}{t^{n+1}+1}\qbinom{n+1}{k}{t^2}
\end{equation}
These are \cite[Lemma 2.1.1]{CZ2} for $(a,b,\lambda)=(k,n+1-k,t^{-1})$ and \cite[Lemma 2.1.2]{CZ2} for $(a,b,\lambda)=(k,n+1-k,t^{-1}),$ respectively.
\end{proof}

\subsection{Straightening laws}\label{strSection}
For $f\in\Z[\Lambda_n],$ we denote
\begin{equation}
    J_n(f)\defeq\sum_{w\in W_n}(-1)^ww(f),
\end{equation}
and for $\lambda\in\Lambda_n$ we consider
\begin{equation}
    \chi_n(\lambda)\defeq\frac{J_n(e^{\lambda+\rho_n})}{J_n(e^{\rho_n})}\in\Z[\Lambda_n],\quad \lambda\in\Lambda_n.
\end{equation}
For $\mu\in\Lambda_n$ we also consider
\begin{equation}
    W_\mu(t,u)\defeq\sum_{\substack{w\in W_n\\ w(\mu)=\mu}}\prod_{\substack{\alpha\in R^+_n\\-w\alpha\in R^+_n}}t_\alpha\in\Z[t,u]
\end{equation}
and the Hall--Littlewood polynomials
\begin{equation}
    P_n(\mu,t,u)\defeq\frac{1}{W_\mu(t,u)}\frac{J_n(e^{\mu+\rho_n}\prod_{\alpha\in R_n^+}(1-t_{\alpha}e^{-\alpha}))}{J_n(e^{\rho_n})}\in\Z[t,u][\Lambda_n]
\end{equation}
and
\begin{equation}
    Q_n(\mu,t,u)\defeq W_\mu(t,u)\cdot P_n(\mu,t,u)\in\Z[t,u][\Lambda_n].
\end{equation}
Recall we also have the inner product
\begin{equation}
    \langle\cdot,\cdot\rangle\colon\Q[t,u][\Lambda_n]\times\Q[t,u][\Lambda_n]\to\Q\llbracket t,u\rrbracket
\end{equation}
given by
\begin{equation}
\langle f,g\rangle=\frac{1}{\# W_n}\mathrm{CT}\left(f\cdot\bar{g}\cdot\prod_{\alpha\in R_n}\frac{1-e^\alpha}{1-t_\alpha e^\alpha}\right)
\end{equation}
where $\bar{g}$ is the linear involution $\bar{e^\lambda}=e^{-\lambda}$ and $\mathrm{CT}$ denotes the coefficient of $e^0$ on the expansion in $\Q[\Lambda_n]\llbracket t,u\rrbracket.$ More generally, we denote $[e^\mu](f)=\mathrm{CT}(e^{-\mu}f)$ the coefficient of $e^\mu.$

\begin{proposition}\label{PQ=1}
    If $\lambda,\mu\in\Lambda^+_n,$ then $\langle P_n(\lambda,t,u),Q_n(\mu,t,u)\rangle=\delta_{\lambda,\mu}.$
\end{proposition}
\begin{proof}
    This is standard. We include the proof for completeness. Using the Weyl denominator formula, the inner product $\langle P_n(\lambda,t,u),Q_n(\mu,t,u)\rangle$ is
    \begin{equation}
    \begin{split}
        &\frac{1}{\# W_n\cdot W_\mu(t,u)}\mathrm{CT}\left(\sum_{w_1,w_2\in W_n}(-1)^{w_1}(-1)^{w_2}e^{w_1(\lambda+\rho_n)-w_2(\mu+\rho_n)}\prod_{\alpha\in R_n^+}\frac{1-t_\alpha e^{-w_1(\alpha)}}{1-t_\alpha e^{-w_2(\alpha)}}\right)\\
        &\qquad=\frac{1}{W_\mu(t,u)}\mathrm{CT}\left(\sum_{w\in W_n}(-1)^{w}e^{\mu+\rho_n-w(\lambda+\rho_n)}\frac{\prod_{\alpha\in R_n^+}1-t_\alpha e^{-\alpha}}{\prod_{\alpha\in R_n^+}1-t_\alpha e^{-w(\alpha)}}\right)\\
        &\qquad=\frac{1}{W_\mu(t,u)}\mathrm{CT}\left(\sum_{w\in W_n}e^{\mu-w(\lambda)}\prod_{\substack{\alpha\in R_n^+\\w^{-1}(\alpha)\in R_n^-}}\frac{t_\alpha-e^\alpha}{1-t_\alpha e^{\alpha}}\right).
    \end{split}
    \end{equation}
        This can only have a constant term when $w(\lambda)-\mu$ is a sum of roots in $R_n^+\cap w(R_n^-),$ but then $w(\lambda)$ is dominant and thus $w(\lambda)=\lambda.$ Finally, for any $\alpha\in R_n^+\cap w(R_n^-)$ we have $\langle\lambda,\alpha^\vee\rangle=0$ and $\langle\mu,\alpha^\vee\rangle\ge0,$ but the condition on $\lambda-\mu$ implies $\langle\lambda-\mu,\alpha^\vee\rangle\ge0.$ Thus we conclude $\lambda=\mu$ and the claim follows.
\end{proof}

We consider the following straightening law for the functions $Q_n(\mu,t,u).$
\begin{proposition}\label{strlaw}
    Let $\beta\in \Delta^+_n$ be a positive simple root and $s_\beta\colon \Lambda_n\to\Lambda_n$ the associated reflection. Then for $\mu\in\Lambda_n$ we have
    \begin{equation}
        Q_n(\mu,t,u)+Q_n(s_\beta(\mu)-\beta,t,u)=t_\beta\left(Q_n(\mu+\beta,t,u)+Q_n(s_\beta(\mu),t,u)\right).
    \end{equation}
\end{proposition}
\begin{proof}
    We have
    \begin{equation}
        Q_n(\mu,t,u)-t_\beta Q_n(\mu+\beta,t,u)=J_n\left(e^{\mu+\rho_n}(1-t_\beta e^\beta)(1-t_\beta e^{-\beta})\prod_{\alpha\in R^+_n\setminus\{\beta\}}(1-t_\alpha e^{-\alpha})\right)/J_n(e^{\rho_n})
    \end{equation}
    and since $R_n^+\setminus\{\beta\}$ is preserved by $s_\beta,$ this is the same as
    \begin{equation}
        -J_n\left(e^{s_\beta(\mu+\rho_n)}(1-t_\beta e^\beta)(1-t_\beta e^{-\beta})\prod_{\alpha\in R^+_n\setminus\{\beta\}}(1-t_\alpha e^{-\alpha})\right)/J_n(e^{\rho_n}).
    \end{equation}
    Since $s_\beta(\mu+\rho_n)=(s_\beta(\mu)-\beta)+\rho_n,$ this is also seen to be
    \begin{equation}
        -(Q_n(s_\beta(\mu)-\beta,t,u)-t_\beta Q_n(s_\beta(\mu),t,u)).
    \end{equation}
    The claim follows.
\end{proof}

\begin{definition}
    We consider the straightening function $\mathrm{str}\colon\Z[t,u][\Lambda_n]\to\Z[t,u][\Lambda^+_n]$ given by
    \begin{equation}
        \mathrm{str}(e^\mu)=\sum_{\lambda\in\Lambda^+_n}e^\lambda\langle P_n(\lambda,t,u),Q_n(\mu,t,u)\rangle
    \end{equation}
    and extended linearly. For $f,g\in\Z[t,u][\Lambda_n],$ we denote $f\equiv g$ if $\mathrm{str}(f-g)=0.$
\end{definition}
Because of \Cref{PQ=1} and \Cref{strlaw}, this means that $\mathrm{str}(e^\mu)=\sum_{\lambda\in\Lambda^+_n} a_{\lambda,\mu}e^\lambda$ for coefficients $a_{\lambda,\mu}\in\Z[t,u]$ which are such that
\begin{equation}
    Q_n(\mu,t,u)=\sum_{\lambda\in\Lambda^+_n}a_{\lambda,\mu}Q_n(\lambda,t,u).
\end{equation}

We now explain how the straightening relations allow us to obtain nontrivial recurrences for the Kostka--Foulkes polynomials, which we will use to prove \Cref{KFthm}.
\begin{proposition}\label{LusztigKatoCor}
    For $\lambda,\mu\in\Lambda^+_n,$ we have $K^{C_n}_{\lambda,\mu}(t,u)=\langle\chi_n(\lambda),Q_n(\mu,t,u)\rangle.$
\end{proposition}
\begin{proof}
    This is standard. See \cite{Kato} for example. We include a proof for completeness. By the Weyl denominator formula, we have
    \begin{equation}
        \langle\chi_n(\lambda),Q_n(\mu,t,u)\rangle=\frac{1}{\# W_n}\mathrm{CT}\left(\frac{J_n(e^{\lambda+\rho_n})J_n(e^{-\mu-\rho_n}\prod_{\alpha\in R_n^+}(1-t_\alpha e^{\alpha}))}{\prod_{\alpha\in R_n}(1-t_\alpha e^\alpha)}\right).
    \end{equation}
    Expanding both $J_n$ terms and combining, this is
    \begin{equation}
        \langle\chi_n(\lambda),Q_n(\mu,t,u)\rangle=\mathrm{CT}\left(\sum_{w\in W_n}(-1)^w e^{w(\lambda+\rho_n)-(\mu+\rho_n)}\prod_{\alpha\in R_n^+}(1-t_\alpha e^{-\alpha})^{-1}\right),
    \end{equation}
    which is the definition of $K^{C_n}_{\lambda,\mu}(t,u).$
\end{proof}
\begin{theorem}[Lusztig--Kato formula]\label{LusztigKato}
    If $\lambda\in\Lambda^+_n,$ we have
    \begin{equation}
        \chi_n(\lambda)=\sum_{\substack{\mu\le\lambda\\\mu\in\Lambda_n^+}}P_n(\mu,t,u)K^{C_n}_{\lambda,\mu}(t,u).
    \end{equation}
\end{theorem}
\begin{proof}
    This follows at once from \Cref{PQ=1} and \Cref{LusztigKatoCor}, since the $P_n(\mu,t,u)$ for $\mu\in\Lambda_n^+$ are a basis of the space of symmetric functions.
\end{proof}

\begin{lemma}\label{KFrec}
    Consider $\lambda\in \Lambda^+_n$ and $\nu\in\Lambda_n.$ Denote $L_n=\prod_{\alpha\in R^{+,l}_n}(1-ue^\alpha).$ Then we have that
    \begin{equation}
        \sum_{\mu\in\Lambda^+_n}K^{C_n}_{\lambda,\mu}(t,u)\cdot [e^\mu](\mathrm{str}(L_ne^\nu))\in \Z[t].
    \end{equation}
    Similarly, denote $S_n=\prod_{\alpha\in R^{+,s}_n}(1-te^\alpha).$ Then we have that
    \begin{equation}
        \sum_{\mu\in\Lambda^+_n}K^{C_n}_{\lambda,\mu}(t,u)\cdot [e^\mu](\mathrm{str}(S_ne^\nu))\in \Z[u].
    \end{equation}
\end{lemma}
\begin{proof}
    Note that
    \begin{equation}
        \sum_{\mu\in\Lambda^+_n}Q_n(\mu,t,u)\cdot [e^\mu](\mathrm{str}(L_ne^\nu))=\sum_{\mu\in\Lambda_n}Q_n(\mu,t,u)\cdot[e^\mu](L_ne^\nu)
    \end{equation}
    and hence this is
    \begin{equation}
    \begin{split}
    \frac{J_n\left(L_ne^\nu e^{\rho_n}\prod_{\alpha\in R^{+}_n}(1-t_\alpha e^{-\alpha})\right)}{J_n(e^{\rho_n})}&=\frac{J_n\left(e^{\nu+\rho_n}\prod_{\alpha\in R^{+,s}_n}(1-te^{-\alpha}) \prod_{\alpha\in R^{+,l}_n}(1-ue^{-\alpha})(1-ue^\alpha)\right)}{J_n(e^{\rho_n})}\\
    &=\frac{J_n\left(e^{\nu+\rho_n}\prod_{\alpha\in R^{+,s}_n}(1-te^{-\alpha}) \right)}{J_n(e^{\rho_n})}\cdot \prod_{\alpha\in R^{l}_n}(1-ue^{\alpha}).
    \end{split}
    \end{equation}
    Thus the first claim follows from \Cref{LusztigKatoCor} by pairing with $\chi_n(\lambda)$: the expression $\sum_{\mu\in\Lambda^+_n}K^{C_n}_{\lambda,\mu}(t,u)\cdot [e^\mu](\mathrm{str}(L_ne^\nu))$ is equal to
    \begin{equation}
        \frac{1}{\# W_n}\mathrm{CT}\left(\overline{\chi_n(\lambda)}\cdot \frac{J_n\left(e^{\nu+\rho_n}\prod_{\alpha\in R^{+,s}_n}(1-te^{-\alpha}) \right)}{J_n(e^{\rho_n})}\cdot \prod_{\alpha\in R^{l}_n}(1-ue^{\alpha})\cdot\prod_{\alpha\in R_n}\frac{1-e^\alpha}{1-t_\alpha e^\alpha}\right)
    \end{equation}
    which is
    \begin{equation}
        \frac{1}{\# W_n}\mathrm{CT}\left(\overline{\chi_n(\lambda)}\cdot \frac{J_n\left(e^{\nu+\rho_n}\prod_{\alpha\in R^{+,s}_n}(1-te^{-\alpha}) \right)}{J_n(e^{\rho_n})}\cdot\frac{\prod_{\alpha\in R_n}(1-e^\alpha)}{\prod_{\alpha\in R_n^s}(1-t e^\alpha)}\right)\in\Q\llbracket t\rrbracket.
    \end{equation}
    
    The second claim follows in the same way.
\end{proof}
The proof of \Cref{KFthm} will be done by induction by repeatedly applying this lemma for different choices of $\lambda$ and for $\nu=0.$ For this, we will need to perform several computations with the straightening laws.

\subsection{Straightening computations}
\begin{definition}
    We consider the concatenation product $\star\colon\Lambda_n\times\Lambda_m\to\Lambda_{n+m},$ and we extend it to $\star\colon \Z[ t,u][\Lambda_n]\times\Z[t,u][\Lambda_m]\to \Z[t,u][\Lambda_{n+m}]$.
\end{definition}
\begin{corollary}\label{strsmall}
    For any $\delta_1\in\Lambda_n$ and $\delta_2\in\Lambda_m$ and $r\ge1,$ we have
    \begin{equation}
        \delta_1\star e^{(0^r,1)}\star\delta_2\equiv\delta_1\star\left(t^re^{(1,0^r)}\right)\star\delta_2
    \end{equation}
    as well as
    \begin{equation}
        \delta_1\star e^{(0^r,2)}\star\delta_2\equiv \delta_1\star \left(t^{r-1}(t^r-1)e^{(1^2,0^{r-1})}+t^re^{(2,0^r)})\right)\star\delta_2.
    \end{equation}
\end{corollary}
\begin{proof}
    The case $r=1$ follows immediately from \Cref{strlaw}, and the general case follows easily from induction on $r.$
\end{proof}
The following lemma will allow us to compute $\mathrm{str}(L_n).$
\begin{lemma}\label{strComputationL}
    Denote
    \begin{equation}
        \delta_{a,n}\defeq\sum_{\substack{\lambda\in\{0,2\}^n\subseteq \Lambda_n\\\lvert\lambda\vert=2a}}e^\lambda.
    \end{equation}
    Then
    \begin{equation}
        [e^{\omega_{2a}}]\mathrm{str}(\delta_{a,n})=(-1)^a\qbinom{n}{2a}{t}(t;t^2)_a.
    \end{equation}
\end{lemma}
\begin{proof}
    Denote $w(a,n)=[e^{\omega_{2a}}]\mathrm{str}(\delta_{a,n}).$ We write
    \begin{equation}
        \delta_{a,n}=\delta_{a,n}^{(0)}+\delta_{a,n}^{(2)}
    \end{equation}
    where for $m\in\{0,2\}$ we define
    \begin{equation}
        \delta_{a,n}^{(m)}\defeq\sum_{\substack{\lambda\in\{0,2\}^n\subseteq\Lambda_n\\\lvert\lambda\rvert=2a\\\lambda_n=m}}e^\lambda.
    \end{equation}
    Then we have
    \begin{equation}
        \mathrm{str}(\delta_{a,n}^{(0)})=\mathrm{str}(\delta_{a,n-1})\star e^{(0)},\quad \delta_{a,n}^{(2)}\equiv \delta_{a-1,n-1}\star e^{(2)}
    \end{equation}
    and thus by \Cref{strsmall} we have the recurrence
    \begin{equation}
        w(a,n)=w(a,n-1)+t^{n-2a}(t^{n-2a+1}-1)\cdot w(a-1,n-1).
    \end{equation}
    This recurrence fully determines all $w(a,n)$ from the starting conditions $w(0,n)=1.$ Thus the proof of the lemma reduces to the identity
    \begin{equation}
        (-1)^a\qbinom{n}{2a}{t}(t;t^2)_a\isequal (-1)^a\qbinom{n-1}{2a}{t}(t;t^2)_a+t^{n-2a}(t^{n-2a+1}-1)(-1)^{a-1}\qbinom{n-1}{2a-2}{t}(t;t^2)_{a-1},
    \end{equation}
    which is straightforward.
\end{proof}

As for $\mathrm{str}(S_n),$ we need a corresponding lemma.
\begin{lemma}\label{strComputationS}
    Let $S_n$ be as in \Cref{KFrec}. Then we have
    \begin{equation}
        [e^{\omega_{2k}}](\mathrm{str}(S_n))\equiv(-1)^kt^{\binom{k+1}{2}}\qbinom{n-k}{k}{t}\mod u\Z[t,u].
    \end{equation}
\end{lemma}
The proof of this lemma is substantially more involved, and will take the rest of this subsection. We advise the reader to skip this proof in a first reading, and instead head to \Cref{KFthmSec} for the proof of \Cref{KFthm}.

\begin{definition}
    We consider the function $p_n\colon\Z[t][\Lambda_n]\to\Z[t]$ given by
    \begin{equation}
        p_n(e^\mu)\defeq[e^{\rho_n}]\left(\sum_{w\in W_n}(-1)^we^{w(\mu+\rho_n)}\right)
    \end{equation}
    and extended linearly.
\end{definition}
\begin{remark}\label{pnRemark}
    Naturally, for $\mu\in\Lambda^+_n$ we have
    \begin{equation}
        p_n(e^{\mu})=\begin{cases}
            1&\text{if }\mu=0,\\
            0&\text{otherwise}.
        \end{cases}
    \end{equation}
\end{remark}
\begin{definition}
    For $1\le m\le n,$ we consider the following elements of $\Z[t][\Lambda_n].$
    \begin{equation}
        T_{m}^-\defeq\prod_{j=m+1}^n(1-te^{-e_m+e_j}),\quad T_m^+\defeq\prod_{j=m+1}^n(1-te^{-e_m-e_j})
    \end{equation}
    and $T_m\defeq T_m^+T_m^-.$
\end{definition}

\begin{proposition}\label{newstr}
    We have
    \begin{equation}
        [e^{\omega_a}](\mathrm{str}(S_n))\equiv\frac{1}{[a]_t![n-a]_t!}\cdot p_n\left(e^{\omega_{a}}\prod_{i=1}^{n}T_{i}\right)\mod u\Z[t,u].
    \end{equation}
\end{proposition}
\begin{proof}
By definition, we have
\begin{equation}
    [e^{\mu}](\mathrm{str}(S_n))\equiv\langle P_n(\mu,t,u),Q_n(S_n,t,u)\rangle\mod u\Z[t,u]
\end{equation}
where we understand $Q_n(f,t,u)$ to be the linear extension of the $Q_n(\mu,t,u).$ Now
\begin{equation}
\begin{split}
    \langle P_n(\mu,t,0),Q_n(S_n,t,0)\rangle&=\frac{1}{\#W_n\cdot W_{\mu}(t,0)}\cdot\mathrm{CT}\left(J_n\left(e^{\mu+\rho_n}\prod_{\alpha\in R^{s,+}_n}(1-te^{-\alpha})\right)J_n(e^{-\rho_n})\right)\\
    &=\frac{1}{W_{\mu}(t,0)}\cdot [e^{\rho_n}]\left(\sum_{w\in W_n}(-1)^ww\left(e^{\mu+\rho_n}\prod_{\alpha\in R^{s,+}_n}(1-te^{-\alpha})\right)\right)\\
    &=\frac{1}{W_{\mu}(t,0)}\cdot p_n\left(e^{\mu}\prod_{\alpha\in R^{s,+}_n}(1-te^{-\alpha})\right).
\end{split}
\end{equation}
And the claim follows from the observation that
\begin{equation}
    \prod_{\alpha\in R^{s,+}_n}(1-te^{-\alpha})=\prod_{i=1}^{n}T_{i}
\end{equation}
and that $W_{\omega_a}(t,0)=[a]_t![n-a]_t!.$
\end{proof}

Rather than computing $p_n\left(e^{\omega_{a}}\prod_{i=1}^{n}T_{i}\right)$ directly, we will instead compute the image of $e^{\omega_{a}}\prod_{i=1}^{n}T_{i}$ under the following equivalence relation. This will ultimately allow us to perform an induction.
\begin{definition}
    We consider the equivalence relation $\equiv_p$ on $\Z[t][\Lambda_n]$ generated by the relations
    \begin{equation}
        e^\mu+e^{s_\beta(\mu+\beta)}\equiv_p0
    \end{equation}
    for all $\beta\in\Delta^+_n.$ Note that this is nothing more than saying that $e^{w(\mu+\rho_n)-\rho_n}\equiv_p (-1)^we^\mu$ for $w\in W_n.$
\end{definition}
\begin{proposition}\label{equivp}
    If $\beta\in\Delta^+_n,$ then
    \begin{equation}
        p_n(e^\mu)=-p_n(e^{s_\beta(\mu+\beta)}).
    \end{equation}
    In particular, if $f\equiv_p g$ then $p_n(f)=p_n(g).$
\end{proposition}
\begin{proof}
    This follows at once from observing that
    \begin{equation}
        \sum_{w\in W_n}(-1)^we^{w(\mu+\rho_n)}=\sum_{w\in W_n}(-1)^{ws_\beta}e^{ws_\beta(\mu+\rho_n)}=-\sum_{w\in W_n}(-1)^we^{w(s_\beta(\mu)-\beta+\rho_n)}.\qedhere
    \end{equation}
\end{proof}

\begin{lemma}\label{pstrcomp}
    Consider $a,b,c,d\ge0.$ Then we have
    \begin{equation}
        e^{(-a,2^b,1^c,0^d)}\equiv_p\begin{cases}
            (-1)^{a+1}e^{(1^{b+c+1},0^d)}&\text{if }b=a+1,\\
            (-1)^ae^{(1^b,0^{c+d+1})}&\text{if }a=b+c,\\
            (-1)^{a+1}e^{(1^b,0^{c+d+1})}&\text{if }a=b+c+2(d+1),\\
            (-1)^{a}e^{(1^{b+c+1},0^{d})}&\text{if }a=b+2c+2d+3\\
            -e^{(a-2(b+c+d+1),2^b,1^c,0^d)}&\text{if }a\ge 2(b+c+d+2),\\
            0&\text{otherwise.}
        \end{cases}
    \end{equation}
\end{lemma}
\begin{proof}
    We will repeatedly use that $e^{\mu}\equiv_p -e^{\mu'}$ if
    \begin{equation}
        \mu'_j=\begin{cases}
        \mu_j&\text{if }j\not\in\{i,i+1\}\\
        \mu_{i+1}-1&\text{if }j=i\\
        \mu_i+1&\text{if }j=i+1.\end{cases}
    \end{equation}
    for some $1\le i<n.$ In particular $e^\mu\equiv_p 0$ if $\mu_{i+1}=\mu_i+1$ for some $1\le i<n.$ Furthermore, note that $e^{\mu}\equiv_p -e^{\mu'}$ if
    \begin{equation}
        \mu'_j=\begin{cases}
        \mu_j&\text{if }j\neq n\\
        2-\mu_j&\text{if }j=n.\end{cases}
    \end{equation}

    Now note that $e^{(-a,2^b,1^c,0^d)}$ is equivalent to $0$ unless $a\ge b-1,$ in which case
    \begin{equation}
        e^{(-a,2^b,1^c,0^d)}\equiv_p(-1)^be^{(1^b,-(a-b),1^c,0^d)}.
    \end{equation}
    So if $a=b-1$ we get $e^{(-a,2^b,1^c,0^d)}\equiv_p(-1)^{a-1}e^{(1^{b+c+1},0^d)}.$ If $a=b,$ then we have
    \begin{equation}
        e^{(-a,2^b,1^c,0^d)}\equiv_p\begin{cases}
            (-1)^ae^{(1^b,0^{d+1})}&\text{if }c=0,\\
            0&\text{if }c>0.
        \end{cases}
    \end{equation}
    If $a>b,$ then $e^{(1^b,-(a-b),1^c,0^d)}$ is equivalent to zero unless $a-b\ge c,$ in which case we have
    \begin{equation}
        (-1)^be^{(1^b,-(a-b),1^c,0^d)}\equiv_p (-1)^{b+c}e^{(1^b,0^c,-(a-b-c),0^d)}.
    \end{equation}
    So if $a=b+c$ we get $e^{(-a,2^b,1^c,0^d)}\equiv_p(-1)^ae^{(1^b,0^{c+d+1})}.$ If $a>b+c,$ then $e^{(1^b,0^c,-(a-b-c),0^d)}$ is equivalent to zero unless $a-b-c\ge d,$ in which case
    \begin{equation}
         (-1)^{b+c}e^{(1^b,0^c,-(a-b-c),0^d)}\equiv_p(-1)^{b+c+d}e^{(1^b,0^c,(-1)^d,-(a-b-c-d))}.
    \end{equation}
    So in fact this is equivalent to zero unless $a-b-c-d\ge2,$ in which case
    \begin{equation}
        (-1)^{b+c+d}e^{(1^b,0^c,(-1)^d,-(a-b-c-d))}\equiv_p(-1)^{b+c+d+1}e^{(1^b,0^c,(-1)^d,(a-b-c-d-2))}.
    \end{equation}
    Continuing similarly, this is equivalent to zero unless $a-b-c-d-2\ge d,$ in which case
    \begin{equation}
        (-1)^{b+c+d+1}e^{(1^b,0^c,(-1)^d,(a-b-c-d-2))}\equiv_p(-1)^{b+c+1}e^{(1^b,0^c,(a-b-c-2d-2),0^d)}.
    \end{equation}
    If $a=b+c+2(d+1),$ this is $(-1)^{a+1}e^{(1^b,0^{c+d+1})}.$ If $a>b+c+2(d+1),$ then this is equivalent to zero unless $a-b-c-2d-2\ge c+1,$ in which case
    \begin{equation}
        (-1)^{b+c+1}e^{(1^b,0^c,(a-b-c-2d-2),0^d)}\equiv_p(-1)^{b+1}e^{(1^b,(a-b-2c-2d-2),1^c,0^d)}.
    \end{equation}
    This is equivalent to zero unless $a-b-2c-2d-2$ is $1$ or $\ge b.$ In the latter case, this is
    \begin{equation}
        (-1)^{b+1}e^{(1^b,(a-b-2c-2d-2),1^c,0^d)}\equiv_p-e^{((a-2b-2c-2d-2),2^b,1^c,0^d)}.
    \end{equation}
    This concludes the proof of the claim.
\end{proof}

\begin{lemma}\label{Scomp}
    Let $0\le k\le n.$ Then
    \begin{equation}
        T_1e^{\omega_{k}}\equiv_p\begin{cases}
            [k]_te^{\omega_{k}}+[n+1-k]_t\sum_{1\le i\le k/2}t^{k-i}([i-1]_tt^{n+i-k+2}-[i+1]_t)e^{\omega_{k-2i}}&\text{if }k>0,\\
            [n]_te^{\omega_0}&\text{if }k=0.
        \end{cases}
    \end{equation}
\end{lemma}
\begin{proof}
    We consider the following auxiliary equivalence relation: $\equiv_p'$ to be generated by the relations
    \begin{equation}
        e^\mu+e^{s_\beta(\mu+\beta)}\equiv_p'0
    \end{equation}
    for all $\beta\in\Delta^{s,+}_n$ with $\beta_1=0.$ We note that if $f\equiv_p' g,$ then $T_1^+f\equiv_p T_1^+g.$

    Thus we start by computing $T_{1}^-e^{\omega_k}\mod\equiv_p'.$ First note that $e^\mu\equiv_p'0$ if $(\mu_i,\mu_{i+1})=(b,b+1)$ for some $2\le i<n$ and $b\in\Z.$ With just this fact, we can see that
    \begin{equation}
        T_1^-e^{\omega_{k}}\equiv_p'\begin{cases}
            \sum_{a=0}^{k-1}\sum_{b=0}^{n-k}(-t)^{a+b}e^{(1-a-b,2^a,1^{k-1-a+b},0^{n-k-b})}&\text{if }k\ge1,\\
            \sum_{i=0}^{n-1}(-t)^ie^{(-i,1^i,0^{n-1-i})}&\text{if }k=0.
        \end{cases}
    \end{equation}
    Similarly, applying $T_1^+$ to both sides give us
    \begin{equation}
        T_1e^{\omega_{k}}\equiv_p\begin{cases}
            \sum_{a=0}^{k-1}\sum_{b=0}^{n-k}\sum_{c=0}^a\sum_{d=0}^{k-1-a+b}(-t)^{a+b+c+d}e^{(1-a-b-c-d,2^{a-c},1^{c+k-1-a+b-d},0^{n-k-b+d})}&\text{if }k\ge1,\\
            \sum_{i=0}^{n-1}\sum_{j=0}^i(-t)^{i+j}e^{(-i-j,1^{i-j},0^{n-1-i+j})}&\text{if }k=0.
        \end{cases}
    \end{equation}
    Hence we can use \Cref{pstrcomp} to compute these.
    
    First note that it immediately follows that
    \begin{equation}
        T_1e^{\omega_0}\equiv_p\sum_{i=0}^{n-1}(-t)^i(-1)^ie^{\omega_0}=[n]_te^{\omega_0}.
    \end{equation}
    Now assume $k\ge1.$ We have the following contributions:
    \begin{itemize}
        \item $(-t)^a(-1)^ae^{(1^{k},0^{n-k})}$ when $b=c=d=0,$ totaling $[k]_te^{\omega_k},$
        \item $-t^{a+b+c+d}e^{(1^{a-c},0^{n-a+c})}$ when $a+c+2d=k,$
        \item $t^{a+b+c+d}e^{(1^{a-c},0^{n-a+c})}$ when $a+2b+c=2n-k+2.$
    \end{itemize}
    For the second point, the appropriate $(a,b,c,d)$ are such that
    \begin{equation}
        a+c+2d=k,\quad 0\le c\le  a<k,\quad 0\le d,\quad 0\le b\le n-k,
    \end{equation}
    hence the contribution is
    \begin{equation}
        -[n+1-k]_t\sum_{a+c+2d=k,\ 0\le c\le a<k,\ 0\le d}t^{k-d}e^{\omega_{k-2c-2d}}.
    \end{equation}
    Writing $i=c+d,$ this is
    \begin{equation}
        -[n+1-k]_t\sum_{i=1}^{\lfloor k/2\rfloor}\sum_{d=0}^it^{k-d}e^{\omega_{k-2i}}=-[n+1-k]_t\sum_{i=1}^{\lfloor k/2\rfloor}t^{k-i}[i+1]_te^{\omega_{k-2i}}.
    \end{equation}
    For the third point, write $i=(n-k-b)+(k-1-a)+2.$ Then $a+2b+c=2n-k+2$ is equivalent to $c=(a+2i-k),$ and $b=n-a-i+1.$ Thus the appropriate $(a,i,d)$ are such that
    \begin{equation}
        2\le i\le k/2,\quad k+1-i\le a<k,\quad 0\le d<n+k-2a-i+1,
    \end{equation}
    hence the contribution is
    \begin{equation}
        \sum_{i=2}^{\lfloor k/2\rfloor}\sum_{a=k+1-i}^{k-1}[n+k-2a-i+1]_tt^{n+1+a+i-k}e^{\omega_{k-2i}}.
    \end{equation}
    Finally, we compute the inner sum:
    \begin{equation}
        \sum_{a=k+1-i}^{k-1}[n+k-2a-i+1]_tt^{n+1+a+i-k}=\frac{t^{n+2}}{t-1}\sum_{a=k+1-i}^{k-1}(t^{n-a}-t^{a-k+i-1})=\frac{t^{n+2}}{t-1}(t^{n-k+1}[i-1]_t-[i-1]_t).
    \end{equation}
    
    The claim follows by summing the three contributions.
\end{proof}

\begin{proof}[Proof of \Cref{strComputationS}]
    By \Cref{newstr}, it suffices to compute
    \begin{equation}
        p_n\left(e^{\omega_{2k}}\prod_{i=1}^nT_i\right).
    \end{equation}
    In fact, we will compute for $0\le a\le n$ that
    \begin{equation}
        e^{\omega_a}\prod_{i=1}^nT_i\equiv_p[a]_t![n-a]_t!\sum_{0\le i\le a/2}(-1)^it^{\binom{i+1}{2}}\qbinom{n-a+i}{i}{t}e^{\omega_{a-2i}},
    \end{equation}
    which implies the claimed result by \Cref{equivp}.

    We note the following: if $\delta\in\Z[t][\Lambda_m],$ $f,g\in\Z[t][\Lambda_{n-m}],$ and $a\ge m,$ then $f\equiv_p g$ implies that $T_{m}(\delta \star g)\equiv_pT_{m}(\delta \star f).$ Furthermore, if $\delta\in\Z[t][\Lambda_{1}],$ $f\in\Z[t][\Lambda_{n-1}]$ and $m\ge2,$ then by definition we have
    \begin{equation}
        T_{m}(\delta\star f)=\delta\star(T_{m-1}f).
    \end{equation}
    Thus we will compute $e^{\omega_a}\prod_{i=1}^{n}T_{i}\mod\equiv_p$ by induction in $a$ with \Cref{Scomp}. The base case $a=0$ follows easily from \Cref{pnRemark}.

    For the induction hypothesis, take $a>0$ and write
    \begin{equation}
        e^{\omega_a}\prod_{i=1}^nT_i\equiv_p[a-1]_t![n-a]_t!\cdot T_1\left(\sum_{i=0}^{\lfloor (a-1)/2\rfloor}(-1)^it^{\binom{i+1}{2}}\qbinom{n-a+i}{i}{t}e^{\omega_{a-2i}}\right),
    \end{equation}
    by the induction hypothesis for $(a-1,n-1).$ By \Cref{Scomp}, it remains to prove that
    \begin{equation}
    \begin{split}
        &\sum_{i=0}^{\lfloor (a-1)/2\rfloor}(-1)^it^{\binom{i+1}{2}}\qbinom{n-a+i}{i}{t}\cdot\\
        &\qquad\cdot\left([a-2i]_te^{\omega_{a-2i}}+[n+1-a+2i]_t\sum_{j\ge1}t^{a-2i-j}([j-1]_tt^{n+j-a+2i+2}-[j+1]_t)e^{\omega_{a-2i-2j}}\right)
    \end{split}
    \end{equation}
    equals $[a]_t\sum_{i\ge0}(-1)^it^{\binom{i+1}{2}}\qbinom{n-a+i}{i}{t}e^{\omega_{a-2i}}.$ Comparing each term, it remains to see that
    \begin{equation}
        \sum_{j=1}^i(-1)^{i-j}t^{\binom{i-j+1}{2}}\qbinom{n-a+i-j}{i-j}{t}[n+1-a+2i-2j]_tt^{a-2i+j}([j-1]_tt^{n-a+2i+2-j}-[j+1]_t)
    \end{equation}
    equals $([a]_t-[a-2i]_t)(-1)^it^{\binom{i+1}{2}}\qbinom{n-a+i}{i}{t}.$ Changing $n$ to $n+a$ and $j$ to $i-j+1,$ as well as dividing by $t^{a-2i},$ it remains to prove the identity
    \begin{equation}
        (-1)^i[2i]_tt^{\binom{i+1}{2}}\qbinom{n+i}{i}{t}\isequal\sum_{j=1}^i(-1)^jt^{\binom{j}{2}}\qbinom{n+j-1}{j-1}{t}[n-1+2j]_tt^{i-j+1}([i-j+2]_t-[i-j]_tt^{n+i+j+1}).
    \end{equation}
    We will do this in the next lemma.
\end{proof}
\begin{lemma}
For all $i\ge0$ and $n\in\Z,$ we have the identity
    \begin{equation}
        (-1)^i[2i]_tt^{\binom{i+1}{2}}\qbinom{n+i}{i}{t}=\sum_{j=1}^i(-1)^jt^{\binom{j}{2}}\qbinom{n+j-1}{j-1}{t}[n-1+2j]_tt^{i-j+1}([i-j+2]_t-[i-j]_tt^{n+i+j+1}).
    \end{equation}
\end{lemma}
\begin{proof}
    In fact, we will prove the following generalization in variables $t,s$:
    \begin{equation}
        (-1)^i[2i]_tt^{\binom{i+1}{2}}\frac{(ts;t)_i}{(t;t)_i}\isequal\sum_{j=1}^i(-1)^jt^{\binom{j}{2}}\frac{(ts;t)_{j-1}}{(t;t)_{j-1}}\frac{st^{2j-1}-1}{t-1}t^{i-j+1}([i-j+2]_t-[i-j]_tst^{i+j+1}).
    \end{equation}
    Note that the claimed identity is the case $s=t^n.$ Furthermore, we note this is a polynomial on $s$ of degree $i.$ Thus for each given $i,$ it suffices to prove this identity for $i+1$ distinct values of $s.$
    
    Consider first the case $s=t^{-k}$ for some $1\le k\le i.$ Then the left hand side is zero, while the right hand side is
    \begin{equation}
        -\sum_{j=1}^{k}t^{(j-k)(j-1)}\qbinom{k-1}{j-1}{t}[2j-1-k]_tt^{i-j+1}([i-j+2]_t-[i-j]_tt^{i+j+1-k}).
    \end{equation}
    This can be seen to be zero since the terms for $j$ and $k-j+1$ sum to zero for each $j.$

    Lastly, it remains to prove the identity for $s=0.$ We need to prove that
    \begin{equation}
        (-1)^i[2i]_tt^{\binom{i+1}{2}}\frac{1}{(t;t)_i}\isequal\sum_{j=1}^i(-1)^jt^{\binom{j}{2}}\frac{1}{(t;t)_{j-1}}\frac{1}{1-t}t^{i-j+1}[i-j+2]_t.
    \end{equation}
    To do so, we consider their generating series. For the left side, it is
    \begin{equation}
    \begin{split}
        L(x,t)&=\sum_{i\ge0}\frac{(-1)^i[2i]_tt^{\binom{i+1}{2}}}{(t;t)_i}x^i=\frac{1}{1-t}\left(\sum_{i\ge0}\frac{t^{\binom{i}{2}}}{(t;t)_i}(-xt)^i-\sum_{i\ge0}\frac{t^{\binom{i}{2}}}{(t;t)_i}(-xt^3)^i\right)=\frac{(xt;t)_\infty-(xt^3;t)_\infty}{1-t}\\
        &=\frac{(xt;t)_\infty}{1-t}\left(1-\frac{1}{(1-xt)(1-xt^2)}\right)=\frac{-xt(1+t-xt^2)}{(1-t)(1-xt)(1-xt^2)}\cdot (xt;t)_\infty
    \end{split}
    \end{equation}
    For the right side, it is
    \begin{equation}
    \begin{split}
        R(x,t)&=\sum_{i\ge0}\left(\sum_{j=1}^i(-1)^jt^{\binom{j}{2}}\frac{1}{(t;t)_{j-1}}\frac{1}{1-t}t^{i-j+1}[i-j+2]_t\right)x^i=\frac{1}{1-t}\sum_{j\ge1}\frac{(-1)^jt^{\binom{j}{2}+1}}{(t;t)_{j-1}}\sum_{i\ge j}t^{i-j}[i-j+2]_tx^i\\
        &=\frac{1}{1-t}\left(\sum_{j\ge0}\frac{(-1)^{j+1}t^{\binom{j+1}{2}+1}}{(t;t)_{j}}x^{j+1}\right)\left(\sum_{i\ge 0}(xt)^i[i+2]_t\right)=\frac{1}{1-t}\left(-xt\cdot (xt;t)_\infty\right)\left(\frac{\frac{1}{1-xt}-\frac{t^2}{1-xt^2}}{1-t}\right)\\
        &=\frac{-xt(1-t^2-xt^2+xt^3)}{(1-t)^2(1-xt)(1-xt^2)}\cdot(xt;t)_\infty=\frac{-xt(1+t-xt^2)}{(1-t)(1-xt)(1-xt^2)}\cdot(xt;t)_\infty.
    \end{split}
    \end{equation}
    Hence $L(x,t)=R(x,t)$ and the proof of the lemma is complete.
\end{proof}

\subsection{Proof of the main theorem}\label{KFthmSec}
We start by noting that
\begin{equation}\label{KFind}
        K^{C_n}_{\omega_{a},\omega_{b}}(t,u)=K^{C_{n-b}}_{\omega_{a-b},0}(t,u).
\end{equation}
We will also denote
\begin{equation}
    K^{C_n}_{\omega_{2k},0}(t,u)=\sum_{a=0}^k(-u)^aK(n,k,a)
\end{equation}
 for $K(n,k,a)\in\Z[t].$ With this, \Cref{KFthm} is equivalent to
\begin{equation}
    K(n,k,a)\isequal \frac{t^{n+1-k}-t^k}{t^{n+1}-1}\left(t^{a^2-2a}\qbinom{n+1}{2a,k-a,n+1-k-a}{t}(t;t^2)_a\right).
\end{equation}

We start by proving a crucial special case.
\begin{proposition}\label{KFthma=0}
    We have
    \begin{equation}
        K^{C_n}_{\omega_{2k},0}(t,0)=\frac{t^{n+1-k}-t^k}{t^{n+1}-1}\qbinom{n+1}{k}{t}.
    \end{equation}
\end{proposition}
\begin{proof}
    We will use the $S$-recurrences of \Cref{KFrec} with $\nu=(0^n)$ and $\lambda=\omega_{2k}.$ By \Cref{strComputationS}, we have
    \begin{equation}
        0=\sum_{i=0}^kK^{C_n}_{\omega_{2k},\omega_{2i}}(t,0)(-1)^it^{\binom{i+1}{2}}\qbinom{n-i}{i}{t}\quad\text{for }k>0.
    \end{equation}
    Thus
    \begin{equation}
        0=\sum_{i=0}^kK(n-2i,k-i,0)(-1)^it^{\binom{i+1}{2}}\qbinom{n-i}{i}{t}\quad\text{for }k>0.
    \end{equation}
    Note that these relations fully determine all $K(n,k,0)$ from the initial condition $K(n,0,0)=1.$ So it remains to verify the identity
    \begin{equation}
        \sum_{i=0}^k\frac{t^{n+1-k-i}-t^{k-i}}{t^{n+1-2i}-1}\qbinom{n+1-2i}{k-i}{t}(-1)^it^{\binom{i+1}{2}}\qbinom{n-i}{i}{t}\isequal 0,\quad\text{for }k>0.
    \end{equation}
    This is
    \begin{equation}
        (t^{n+1-k}-t^k)\left(\sum_{i=0}^k\frac{1}{t^{n+1-i}-1}\qbinom{n+1-i}{k-i,i,n+1-k-i}{t}(-1)^it^{\binom{i}{2}}\right)
    \end{equation}
    which for $k>0$ we may write as
    \begin{equation}
        \frac{(t^{n+1-k}-t^k)}{t^k-1}\left(\sum_{i=0}^k\qbinom{k}{i}{t}\qbinom{n-i}{k-1}{t}(-1)^it^{\binom{i}{2}}\right).
    \end{equation}
    Now we note that this inner sum is identically zero. For this, we write $\qbinom{n-i}{k-1}{t}$ as a polynomial in $t^{n-i}$
    \begin{equation}
        \qbinom{n-i}{k-1}{t}=\sum_{a=0}^{k-1}t^{(n-i)a}P_a(t)
    \end{equation}
    where $P_a(t)\in\Z[t,t^{-1}]$ do not depend on $i.$ The expression is clearly zero due to the $t$-binomial theorem
    \begin{equation}
        \sum_{i=0}^k\qbinom{k}{i}{t}(-1)^it^{\binom{i}{2}}t^{(n-i)a}=\prod_{i=0}^{k-1}(t^a-t^i)=0\quad\text{for }0\le a\le k-1.\qedhere
    \end{equation}
\end{proof}
\begin{proof}[Proof of \Cref{KFthm}]
    We will use the $L$-recurrences of \Cref{KFrec} with $\nu=(0^n)$ and $\lambda=\omega_{2k}.$ 
   
    Denoting $L_n=\prod_{\alpha\in R^{+,l}_n}(1-ue^\alpha)$ and 
    \begin{equation}
        \delta_{a,n}\defeq\sum_{\substack{\lambda\in\{0,2\}^n\subseteq \Lambda_n\\\lvert\lambda\vert=2a}}e^\lambda.
    \end{equation}
    as before, we note that
    \begin{equation}
        L_ne^{(0^n)}=\sum_{a\ge0}(-u)^a\delta_{a,n}.
    \end{equation}
    Hence, by \Cref{KFrec} with $\nu=(0^n)$ and $\lambda=\omega_{2k},$ together with \Cref{strComputationL}, we have that
    \begin{equation}
        \sum_{i=0}^k K^{C_n}_{\omega_{2k},\omega_{2i}}(t,u)\cdot u^i\qbinom{n}{2i}{t}(t;t^2)_i\in\Z[t].
    \end{equation}
    This is
    \begin{equation}
        \sum_{i=0}^k\sum_{a\ge0}(-u)^aK(n-2i,k-i,a)\cdot u^i\qbinom{n}{2i}{t}(t;t^2)_i\in\Z[t],
    \end{equation}
    which we may write as
    \begin{equation}
        \sum_{a\ge0}(-u)^a\sum_{i\ge0}K(n-2i,k-i,a-i)(-1)^i\qbinom{n}{2i}{t}(t;t^2)_i\in\Z[t].
    \end{equation}
    In other words, we obtain the recurrence relations
    \begin{equation}
        \sum_{i\ge0}K(n-2i,k-i,a-i)(-1)^i\qbinom{n}{2i}{t}(t;t^2)_i=0\quad\text{for }a>0.
    \end{equation}
    Now we note these uniquely determine all the $K(n,k,a)$ from the initial condition of the $K(n,k,0),$ which was verified in \Cref{KFthma=0}. Thus it remains to verify the identity
    \begin{equation}
        \sum_{i\ge0}\frac{t^{n+1-k-i}-t^{k-i}}{t^{n+1-2i}-1}\left(t^{(a-i)^2-2(a-i)}\qbinom{n+1-2i}{2(a-i),k-a,n+1-k-a}{t}(t;t^2)_{a-i}\right)(-1)^i\qbinom{n}{2i}{t}(t;t^2)_i\isequal 0\quad\text{for }a>0.
    \end{equation}
    Indeed, this is
    \begin{equation}
        \frac{t^{n+1-k-a}-t^{k-a}}{t^{n+1}-1}\sum_{i\ge0}\left((-1)^it^{(a-i)^2-(a-i)}\qbinom{n+1}{2(a-i),k-a,n+1-k-a,2i}{t}(t;t^2)_{a-i}(t;t^2)_i\right),
    \end{equation}
    which is
    \begin{equation}
        \frac{t^{n+1-k-a}-t^{k-a}}{t^{n+1}-1}\qbinom{n+1}{k-a,n+1-k-a,2a}{t}(t;t^2)_a\sum_{i\ge0}(-1)^i(t^2)^{\binom{a-i}{2}}\qbinom{a}{i}{t^2},
    \end{equation}
    and the $t^2$-binomial theorem tells us that
    \begin{equation}
        \sum_{i\ge0}(-1)^i(t^2)^{\binom{a-i}{2}}\qbinom{a}{i}{t^2}=\begin{cases}
            1&\text{if }a=0,\\
            0&\text{if }a>0.
        \end{cases}\qedhere
    \end{equation}
\end{proof}
\begin{corollary}\label{ThmA}
    For $0\le b\le a\le n$ with $a\equiv b\mod 2,$ we have
    \begin{equation}
        K^{C_n}_{\omega_{a},\omega_b}(t,u)=t^{\frac{a-b}{2}}\frac{t^{n+1-a}-1}{t^{n+1-b}-1}\sum_{i=0}^{\frac{a-b}{2}}(-u)^i\left(t^{i^2-2i}\qbinom{n-b+1}{2i,\frac{a-b}{2}-i,n+1-\frac{a+b}{2}-i}{t}(t;t^2)_i\right).
    \end{equation}
\end{corollary}
\begin{proof}
    This follows at once from \eqref{KFind} and \Cref{KFthm}.
\end{proof}

\section{Explicit Lusztig--Kato formula for non-split orthogonal groups}\label{SectionLK}
Let $F/\Q_p$ be a finite extension with residue field of cardinality $q,$ and let $V$ be a quadratic space over $F$ of discriminant $1\in\Q_p^\times/\Q_p^{\times2},$ Hasse invariant $-1$ and dimension $2n+2$ for some $n\ge1.$ Consider the group $G=SO(V)$ with special subgroup $K$ and split torus $A.$

We consider the root system of the little Weyl group $W,$ which is the root system $R_n$ considered in \Cref{KFsetup}. We have the Cartan decomposition
\begin{equation}
    K\backslash G/K=A/A\cap K.
\end{equation}
We consider the ($\Z$-valued) Hecke algebra $\mathcal{H}(G,K)=C_c^\infty(K\backslash G/K,\Z)$ under convolution.
\begin{definition}
    For $1\le i\le n,$ we denote $T_i\in\mathcal{H}(G,K)$ the operator corresponding to $\varpi^{\omega_i}\in A$ where $\omega_i=(\underbrace{1,\ldots,1}_i,\underbrace{0,\ldots,0}_{n-i})$ are the fundamental weights. We also denote $T_0\in\mathcal{H}(G,K)$ the identity operator.
\end{definition}

Denoting $E/F$ the unramified quadratic extension and $\mathrm{Gal}(E/F)=\{1,\sigma\},$ the Langlands dual ${}^LG=\check{G}\rtimes\{1,\sigma\}$ of $G$ is identified with the split orthogonal group ${}^LG=O(2n+2)$ where $\sigma\in O(2n+2)$ is the reflection along the last root. We may choose the usual basis with quadratic form $x_1x_{2n+2}+x_2x_{2n+1}+\cdots+x_{n+1}x_{n+2},$ so that the maximal torus $\check{T}$ is the diagonal subgroup $\check{T}=\{\mathrm{diag}(t_1,\ldots,t_{n+1},t_{n+1}^{-1},\ldots,t_1^{-1})\colon t_1,\ldots,t_{n+1}\in \C^\times\}.$
\begin{definition}
    We consider the Satake transform
    \begin{equation}
        \mathrm{Sat}\colon \mathcal{H}(G,K)\to \C[X^\bullet(\check{T})^\sigma]^W=\C[\boldsymbol{\mu}_1,\ldots,\boldsymbol{\mu}_n]^{\mathrm{sym}}
    \end{equation}
    where $\boldsymbol{\mu}_i=\mu_i+\mu_i^{-1}\in\C[X^\bullet(\check{T})^\sigma]^{W}$ for $\mu_i\in X^\bullet(\check{T})$ the character of $\check{T}$ which is projection onto the $i$-th entry.
\end{definition}
\begin{theorem}[Explicit Lusztig--Kato formula]\label{ThmB}
    We denote $\boldsymbol{\mu}=(\boldsymbol{\mu}_1,\ldots,\boldsymbol{\mu}_n)$, and denote by $s_i(\boldsymbol{\mu})$ its $i$-th elementary symmetric polynomial. For $0\le\delta\le n,$ we have
    \begin{equation}
        \sum_{\substack{\delta\le i\le n\\i\equiv\delta\mod2}}\frac{q^{\delta+1}-1}{q^{i+1}-1}\qbinom{i+1}{\frac{i-\delta}{2}}{q^2}\mathrm{Sat}(T_{n-i})=q^{\frac{1}{2}(n^2+n-\delta^2-\delta)}\sum_{\substack{\delta\le i\le n\\i\equiv\delta\mod2}}\frac{\delta+1}{i+1}\binom{i+1}{\frac{i-\delta}{2}}s_{n-i}(\boldsymbol{\mu}).
    \end{equation}
\end{theorem}
\begin{proof}
    By the Macdonald formula for quasi-split groups deduced from \cite{Casselman}, we have that
    \begin{equation}
        \mathrm{Sat}(T_{n-i})=q^{\langle\omega_{n-i},\rho_n\rangle}P_n(\omega_{n-i},q^{-1},q^{-2})=q^{\frac{1}{2}(n^2+n-i^2-i)}P_n(\omega_{n-i},q^{-1},q^{-2})
    \end{equation}
    for the polynomial $P_n(\omega_{n-i},t,u)$ defined in \Cref{strSection}. Thus by \Cref{LusztigKato} we have
    \begin{equation}
        \sum_{\delta\le i\le n}q^{-\frac{1}{2}(n^2+n-i^2-i)}K^{C_n}_{\omega_{n-\delta},\omega_{n-i}}(q^{-1},q^{-2})\cdot \mathrm{Sat}(T_{n-i})=\chi_n(\omega_{n-\delta}).
    \end{equation}
    By \Cref{KFcor}, the nonzero terms in the above sum are the ones with $i\equiv\delta\mod 2,$ in which case
    \begin{equation}
        K^{C_n}_{\omega_{n-\delta},\omega_{n-i}}(q^{-1},q^{-2})=K^{C_{i}}_{\omega_{i-\delta},\omega_0}(q^{-1},q^{-2})=\frac{q^{\frac{i-\delta}{2}-(i+1)}-q^{-\frac{i-\delta}{2}}}{q^{-(i+1)}-1}\qbinom{i+1}{\frac{i-\delta}{2}}{q^{-2}}.
    \end{equation}
    This is
    \begin{equation}
        K^{C_n}_{\omega_{n-\delta},\omega_{n-i}}(q^{-1},q^{-2})=\frac{q^{\delta+1}-1}{q^{i+1}-1}\qbinom{i+1}{\frac{i-\delta}{2}}{q^2}q^{-\frac{1}{2}(i^2+i-\delta^2-\delta)}.
    \end{equation}

    Thus it remains to see that
    \begin{equation}
        \chi_n(\omega_{n-\delta})\isequal \sum_{\delta\le i\le n}\frac{\delta+1}{i+1}\binom{i+1}{\frac{i-\delta}{2}}s_{n-i}(\boldsymbol{\mu}).
    \end{equation}
    Note that the representation of $\check{G}=SO(2n+2)$ of highest weight $\omega_{n-\delta}$ is simply $W_{n-\delta}\defeq\bigwedge^{n-\delta}\mathrm{std},$ which extends to an irreducible representation of ${}^LG=O(2n+2).$ Thus by \cite[Section 2.3]{Wendt}, we have that $\chi_n(\omega_{n-\delta})$ is the character of $\check{T}\sigma$ in $W_{n-\delta}.$ That is,
    \begin{equation}
        \chi_{W_{n-\delta}}(t\sigma)=\sum_{\substack{I\subseteq[2n+2]\setminus\{n+1,n+2\}\\\lvert I\rvert=n-\delta}}\prod_{i\in I}\mu_i(t)-\sum_{\substack{\{n+1,n+2\}\subseteq I\subseteq[2n+2]\\\lvert I\rvert=n-\delta}}\prod_{i\in I}\mu_i(t)
    \end{equation}
    and thus we are looking at
    \begin{equation}
        \chi_n(\omega_{n-\delta})=s_{n-\delta}(\mu_1,\mu_1^{-1},\ldots,\mu_{n},\mu_{n}^{-1})-s_{n-\delta-2}(\mu_1,\mu_1^{-1},\ldots,\mu_{n},\mu_{n}^{-1}).
    \end{equation}
    This is
    \begin{equation}
    \begin{split}
        &\sum_{0\le i\le n-\delta}\binom{\delta+i}{i/2}s_{n-\delta-i}(\boldsymbol{\mu})-\sum_{0\le i\le n-\delta-2}\binom{\delta+i+2}{i/2}s_{n-\delta-2-i}(\boldsymbol{\mu})\\
        &\qquad=\sum_{\delta\le i\le n}\binom{i}{\frac{i-\delta}{2}}s_{n-i}(\boldsymbol{\mu})-\sum_{\delta+2\le i\le n}\binom{i}{\frac{i-\delta}{2}-1}s_{n-i}(\boldsymbol{\mu}),
    \end{split}
    \end{equation}
    and since
    \begin{equation}
        \binom{n}{m}-\binom{n}{m-1}=\binom{n+1}{m}\frac{n-2m+1}{n+1}\quad\text{for }n,m\ge0,
    \end{equation}
    we conclude that indeed
    \begin{equation}
        \chi_n(\omega_{n-\delta})=\sum_{\delta\le i\le n}\frac{\delta+1}{i+1}\binom{i+1}{\frac{i-\delta}{2}}s_{n-i}(\boldsymbol{\mu}).\qedhere
    \end{equation}
\end{proof}

\begingroup
\bibliographystyle{alpha}
\bibliography{structural/references}
\endgroup
\end{document}